\documentclass[12pt,a4paper]{amsart}
\usepackage[english]{babel}
\usepackage{amsmath}
\usepackage{amsthm}
\usepackage{amssymb}
\usepackage{amscd}

\newtheorem{theorem}{Theorem}[section]
\newtheorem{definition}{Definition}
\newtheorem{lemma}[theorem]{Lemma}
\newtheorem{proposition}[theorem]{Proposition}
\newtheorem{corollary}[theorem]{Corollary}
\newtheorem{remark}[theorem]{Remark}

\newtheorem{example}[theorem]{Example}

\newtheorem{question and remark}[theorem]{Question and Remark}
\newtheorem{setup}[theorem]{Setup}
\newtheorem{problems}[theorem]{Problems}

\def\Ext {\mathop{\rm Ext}\nolimits}
\def\Soc {\mathop{\rm Soc}\nolimits}
\DeclareMathOperator{\Hom}{Hom}

\DeclareMathOperator{\Ker}{Ker}

\DeclareMathOperator{\Ass}{Ass}
\DeclareMathOperator{\ann}{ann}
\DeclareMathOperator{\cht}{cht}
\DeclareMathOperator{\hit}{ht}
\DeclareMathOperator{\Specm}{Specm}
\DeclareMathOperator{\Spec}{Spec}

\DeclareMathOperator{\imm}{Im}
\renewcommand\Im{\imm}

\DeclareMathOperator\Min{Min}
\DeclareMathOperator\Max{Max}

\DeclareMathOperator\Supp{Supp}

\DeclareMathOperator\Mod{Mod}

\DeclareMathOperator{\gl}{gl} \DeclareMathOperator{\ssl}{sl}
\newcommand\Ga{\Gamma}
\newcommand\bZ{\mathbb{Z}}

\begin{document}

\title[Torsion theories induced from commutative subalgebras
\  ]{Torsion theories induced from commutative subalgebras}
\author{Vyacheslav Futorny }
\author{Serge Ovsienko  }
\author{Manuel Saorin}
\address{Instituto de Matem\'atica e Estat\'\i stica,
Universidade de S\~ao Paulo, Caixa Postal 66281, S\~ao Paulo, CEP
05315-970, Brasil} \email{futorny@ime.usp.br}
\address{
Faculty of Mechanics and Mathematics, Kiev Taras Shevchenko
University, Vla\-di\-mir\-skaya 64, 00133, Kiev, Ukraine}
\email{ovsienko@zeos.net }
\address{Departamento de Matem\'aticas, Universidad de Murcia, 30100 Espinardo, Murcia, Spain}
\email{msaorinc@um.es} 

\date{}

\thanks{
}
\begin{abstract}
We begin a  study of torsion theories for representations of an
important class of associative algebras over a field which includes
all finite $W$-algebras of type $A$, in particular the universal
enveloping algebra of $\gl_n$ (or $\ssl_n$) for all $n$. If $U$ is
such and algebra which contains  a finitely generated commutative
subalgebra $\Gamma$,  then we show that any $\Gamma$-torsion theory
defined by the coheight of prime ideals is liftable to $U.$
Moreover,
 for any simple $U$-module $M$,  all associated prime ideals of $M$ in $\Spec \Gamma$ have the same coheight. Hence,
 the
  coheight of the associated prime ideals of $\Gamma$ is an invariant of a given simple
 $U$-module. This implies a stratification of the category of $U$-modules controlled by the coheight
 of associated prime ideals
 of $\Gamma$.
Our approach can be viewed as a generalization of the classical
paper by R.Block \cite{bl}, it allows in particular to study
representations of $\gl_n$ beyond the classical category of weight
or generalized weight modules.

\end{abstract}
\maketitle

\section{Introduction}

A classical, very difficult and intriguing problem in
representation theory of Lie algebras is the classification of
simple modules over  complex simple finite dimensional Lie
algebras. Such a classification is only known for the Lie algebra
$\ssl_2$ due to results of R.Block \cite{bl}. It remains an open
problem in general, even in the subcategory of weight  modules
with respect to a fixed Cartan subalgebra. On the other hand, a
classification of simple weight  modules  with finite dimensional
weight spaces is well known for any simple finite dimensional Lie
algebra, due to Fernando \cite{Fe} and Mathieu \cite{Ma}.

The basic idea, proposed in \cite{bl} in the case of $\ssl_{2}$ can
be explained as follows. First, we consider a maximal commutative
subalgebra $\Ga\subset U(\ssl_{2})$ (in our terms, Gelfand-Tsetlin
subalgebra), which is generated by a Cartan subalgebra  and the
center of $U(\ssl_{2})$. Then one fixes a central character $\chi$
of $U(\ssl_{2})$. After that all simples modules with central
character $\chi$ are divided into torsion (or generalized weight)
and torsionfree modules with respect to $\Ga/(\Ker \chi)$.
Thereafter the investigation of both classes of modules is reduced
to the investigation of simples over a (skew) group algebra of the
group $\bZ$. An analogous idea works in the more general context of
generalized Weyl algebras of rank $1$ (\cite{ba}, \cite{bavo}),
which allow a complete classification of simple modules.

A similar approach applied in the case of a Lie algebra $\gl(n)$ (or
$\ssl_n$) allows to go beyond the category of weight modules with
finite dimensional
 spaces. Namely, one considers
 the full subcategory of
weight Gelfand-Tsetlin  $\gl_n$-modules with respect to the
\emph{Gelfand-Tsetlin subalgebra} (certain maximal commutative
subalgebras of $U(\gl_n)$)  \cite{dfo:gz}, \cite{fo-Ga2}, that is,
those modules $V$ that have a decomposition
$$V=\oplus_{{\bf m}\in \Specm \Gamma} V({\bf m}),$$ where $V({\bf
m})=\{v\in V| \exists N, {\bf m}^N v=0\}$ as $\Gamma$-modules.
This class is based on natural properties of a Gelfand-Tsetlin
basis for finite dimensional representations of simple classical
Lie algebras \cite{g-ts},  \cite{zh:cg}, \cite{m:gtsb}.
Gelfand-Tsetlin subalgebras were considered in various connections
in
 \cite{FM}, \cite{Vi}, \cite{KW-1}, \cite{KW-2}, \cite{Gr}.
The theory developed in \cite{fo-Ga1} and \cite{fo-Ga2} was an
attempt to unify the representation theories of the universal
enveloping algebra of $\gl_n$ and of the generalized Weyl algebras.
We underline that Gelfand-Tsetlin modules over $\gl_n$ are weight
modules with respect to some Cartan subalgebra of $\gl_n$ but they
are allowed to have infinite dimensional weight spaces.

  In this paper we begin a  study of general torsion theories for
representations of a certain class of associative algebras which
includes all finite $W$-algebras of type $A$. In particular, the
universal enveloping algebra of $\gl_n$ (or $\ssl_n$) is an
example of such algebra for all $n$, where $\Gamma$ is a
Gelfand-Tsetlin subalgebra.

In the rest of the paper we shall work over an algebraically
closed field $K$ of characteristic zero and consider the following
situation.

\begin{setup} \label{setup}
  $U$ will be a $K$-algebra  having  a  commutative (not necessarily
central) subalgebra $\Gamma$, fixed from now on,  satisfying the
following properties:

\begin{enumerate}
\item $\Gamma$ is finitely generated as a $K$-algebra \item There
is a finite subset $\{u_1,\ldots,u_n\}\subset U$ such that $U$ is
generated as a $K$-algebra by $\Gamma\cup\{u_1,\ldots,u_n\}$ \item
$\Gamma$ is a \emph{Harish-Chandra subalgebra}, i.e., for each $u\in
U$ the $\Gamma$-bimodule $\Gamma u\Gamma$ is a finitely generated
$\Gamma$-module both on the left and on the right.
\end{enumerate}
\end{setup}

If $M$ is a Gelfand-Tsetlin $U$-module with respect to $\Gamma$ then
the associated prime ideals of $V$ in $\Spec \Ga$ which form the
assassin $\Ass(M)$ are maximal. Our goal is to understand torsion
categories of modules over $U$ more general than Gelfand-Tsetlin
categories. Such modules have associated primes in $\Spec \Ga$ which
are not maximal.

Our main result is the following theorem. We refer to
Section~\ref{sec-torsion} for definitions.

\begin{theorem} \label{theorem A}
  Let $\Gamma$ be a finitely generated
subalgebra  and $U\supset \Gamma$ as above. Then
\begin{enumerate}
\item
 The $\Ga$-torsion theory associated to the subset $Z_i\subset Spec(\Gamma )$ of prime ideals of coheihgt $\leq i$ is liftable to $U.$
\item For any simple $U$-module $M$ all associated prime ideals of
$M$ in $\Spec \Ga$ have the same height.
\end{enumerate}
\end{theorem}

Theorem A provides a stratification of the module category with
respect to the coheight of the associated primes. In classical cases
as finite $W$-algebras
it happens that the endomorphism algebra of any simple $U$-module
is one dimensional  and the center $Z=Z(U)$ of $U$ is an integral
domain (polynomial ring) contained in $\Gamma$, which is in turn
is also an integral domain (polynomial ring) and flat over $Z$.
Under these circumstances (see Proposition \ref{final
proposition},
  all simple objects in the module category $U-Mod$ are
exhausted by simple $U$-modules  whose associated primes have a
fixed coheight $0\leq i\leq Kdim(\Gamma )-Kdim(Z)$, where $Kdim$
denotes the Krull dimension. The case $i=0$ corresponds to
Gelfand-Tsetlin modules (with respect to $\Gamma$) and the case
$i=Kdim(\Gamma )-Kdim(Z)$ corresponds to the  simple $U$-modules
which are torsionfree with respect to some  central character
$\chi :Z\longrightarrow K$.

Our second main result provides  information about the assassin of a
simple $U$-module.

\begin{theorem}
 Let $U$, $\Gamma$, $u_1, \ldots, u_n$ be as in Setup~\ref{setup},
$M=Ux$  a cyclic  $U$-module generated by an element $x$ such that
$\ann_\Gamma (x)=\mathbf{p}$ is a prime ideal of $\Gamma$  and
suppose that all ideals in $\Ass (M)$ have the same coheight. If
$\mathbf{q}\in \Ass (M)$ then there is a sequence
$\mathbf{q}=\mathbf{q}_0,\mathbf{q}_1,\ldots,
\mathbf{q}_s=\mathbf{p}$ of prime ideals with coheight equal to the
coheight of $\mathbf{p}$ and a sequence of indices $k_1,\ldots,
k_s\in\{1,\ldots, n\}$ such that

\begin{center}
$\frac{\Gamma u_{k_i}\Gamma}{\mathbf{q}_{i-1}u_{k_i}\Gamma +\Gamma
u_{k_i}\mathbf{q}_i}\neq 0$,
\end{center}
for all $i=1,\ldots, s$.
\end{theorem}

All these results can be applied to the class of Galois orders
over finitely generated Noetherian domains \cite{fo-Ga1}. In
particular,  the results are valid for all finite $W$ algebras of
type $A$, e.g. $U(\gl_n)$ for all $n$.

\section{Torsion theories over a commutative Noetherian
ring}\label{sec-torsion}

In this section we collect some  facts concerning torsion theories
over commutative Noetherian rings. Recall that, given a not
necessarily commutative ring $R$, a \emph{torsion theory} over $R$ is
a pair $(\mathcal T,\mathcal F)$ of full subcategories of $R-\Mod$
satisfying the following two conditions:

\begin{enumerate}
\item $\mathcal T =^\perp\mathcal F$ consists of those $R$-module
$T$ such that $\Hom_R(T,F)=0$, for all $F\in\mathcal F$ \item
$\mathcal F =\mathcal T^\perp$ consists of those $R$-module $F$ such
that $\Hom_R(T,F)=0$, for all $T\in\mathcal T$
\end{enumerate}

Note that any of the component class of a torsion theory determines
the other. In the above situation, for every $R$-module $M$ there
exists a (unique up to isomorphism) exact sequence

\begin{center}
$0\rightarrow T\longrightarrow M\longrightarrow F\rightarrow 0$,
\end{center}
with $T\in\mathcal T$ and $F\in\mathcal F$.  Then the assignments
$M\rightsquigarrow t(M):=T$ and $M\rightsquigarrow F=:M/t(M)$ are
functorial and yield a right adjoint and a left adjoint,
respectively, to the inclusion functors $\mathcal T\hookrightarrow
R-\Mod$ and $\mathcal F\hookrightarrow R-\Mod$. The functor
$t:R-\Mod\longrightarrow\mathcal{T}$ is called the \emph{torsion
radical} associated to $\mathcal{T}$. The torsion theory is called
\emph{hereditary} when $\mathcal T$ is closed under taking
submodules, which is equivalent to say that $\mathcal F$ is closed
under taking injective envelopes (see chapter VI of \cite{St} for
all details and terminology concerning torsion theories).

In this paper we are mainly interested in torsion theories over
commutative Noetherian rings. In this section, unless otherwise
stated, $\Gamma$ \emph{will be a commutative Noetherian ring}, We
shall denote by $\Spec \Gamma$ (resp. $\Specm  \Gamma$) the prime
(resp. maximal) spectrum of $\Gamma$.  Given a $\Gamma$-module $M$
and a prime ideal $\mathbf{p}\in \Spec\Gamma$, we shall denote by
$M_{\mathbf{p}}$ the localisation of $M$ at $\mathbf{p}$. We shall
consider two important subsets of $\Spec\Gamma$ associated to $M$.
Namely the \emph{support} of $M$,
 $\Supp (M)=\{\mathbf{p}\in
\Spec \Gamma |\, M_{\mathbf{p}}\neq 0\}$, and the so-called
\emph{assassin} of $M$, $\Ass (M)$, which consists of those
$\mathbf{p}\in \Spec(\Gamma )$ such that $\mathbf{p}=\ann_\Gamma
(x):=\{g\in\Gamma : $ $gx=0\}$, for some $x\in M$.

We now recall some properties of these sets. In the statement and in
the sequel, we denote by $\Min X$ (resp. $\Max X$) the set of minimal
(resp. maximal) elements of $X$, for every subset $X\subset
\Spec\Gamma$.

\begin{proposition} \label{minimal_in_Ass_and_Supp }
Let $X\subseteq \Spec\Gamma$ be any nonempty subset and $M$ be a
$\Gamma$-module. The following assertions hold:

\begin{enumerate}
\item Every element of $X$ contains a minimal element of $X$ \item
$\Ass (M)\subseteq \Supp (M)$ and $\Min \Ass (M)=\Min \Supp (M)$.
\end{enumerate}
\end{proposition}
\begin{proof}
The set $\Spec\Gamma$ satisfies DCC with respect to inclusion. Indeed
if $\mathbf{p}=\mathbf{p}_0\supseteq\mathbf{p}_1\supseteq ...$ is a
descending chain of prime ideals, then the number of nonzero terms
in it is bounded above by the height of $\mathbf{p}$, which is
always finite (cf. \cite{mat}[Theorem 13.5]).

If $X\subseteq \Spec\Gamma$ is any nonempty subset and $\mathbf{p}\in
X$, then, by the DCC property, the set $\{\mathbf{q}\in X :$
$\mathbf{q}\subseteq\mathbf{p}\}$ has a minimal element  which is
turn a minimal element of $X$.

Let now take $\mathbf{p}\in \Ass (M)$, so that $\mathbf{p}=\ann_\Gamma
(\Gamma x)$, for some $x\in M$. Then $\mathbf{p}\in \Ass (\Gamma
x)\subseteq \Supp (\Gamma x)$ (see \cite{mat}[Theorem 6.5]). Putting
$N=\Gamma x$, we get that $N_\mathbf{p}\neq 0$, which implies that
$M_\mathbf{p}\neq 0$ due to the exactness of localization. Then
$\Ass (M)\subseteq \Supp (M)$.

Since $M$ is the directed union of its finitely generated submodules
and localization is exact and preserves direct unions it follows
that $\Supp (M)=\bigcup_{N<M}\Supp (N)$, where the union is taken over
all finitely generated submodules $N$ of $M$. In particular, if
$\mathbf{p}\in \Min \Supp (M)$ then $\mathbf{p}\in \Min \Supp (N)$, for some
$N<M$ finitely generated. But then $\mathbf{p}\in \Ass (N)$ (cf.
\cite{mat}[Theorem 6.5]), and so $\mathbf{p}\in \Ass (M)$. From the
inclusion $\Ass (M)\subseteq \Supp (M)$ we conclude that $\mathbf{p}\in
\Min \Ass (M)$, thus proving that $\Min \Supp (M)\subseteq \Min \Ass (M)$.

Conversely, if $\mathbf{p}\in \Min \Ass (M)$ then we fix a cyclic
submodule $N=\Gamma x$ such that $\mathbf{p}=\ann_\Gamma (N)$. Then
we have $\mathbf{p}\in \Ass (N)\subset \Supp (N)\subset \Supp (M)$. By
assertion 1, there exists $\mathbf{q}\in \Min \Supp (M)$ such that
$\mathbf{q}\subseteq\mathbf{p}$. But equality must hold since we
already know that $\Min \Supp (M)\subseteq \Min \Ass (M)$ and $\mathbf{p}$
is minimal in $\Ass (M)$. Therefore $\mathbf{p}\in \Min \Supp (M)$ and we
get that $\Min \Ass (M)=\Min \Supp (M)$.

\end{proof}

\begin{definition}
A subset $Z\subseteq \Spec \Gamma$ is called \emph{closed under
specialization} when the following property holds:

(*) If $\mathbf{p}\subseteq\mathbf{q}$ are prime ideals with
$\mathbf{p}\in Z$, then $\mathbf{q}$ belongs to $Z$.
\end{definition}

The prototypical examples of  closed under specialization subsets of
$\Spec \Gamma$ are the Zariski-closed subsets  and those  of the form
$\Supp (M)$, where $M$ is a $\Gamma$-module. The following is a
crucial result from \cite{Ga}.

\begin{theorem} \label{Gabriel_bijection}
Let $\Gamma$ be a commutative Noetherian ring. The assignments
$Z\rightsquigarrow (\mathcal T_Z,\mathcal T_Z^\perp )$, where
$\mathcal T_Z=\{T\in \Gamma -\Mod:$ $\Supp (T)\subseteq Z\}$,  and
$(\mathcal T,\mathcal F)\rightsquigarrow Z_{(\mathcal T,\mathcal
F)}=\{\mathbf{p}\in \Spec \Gamma :$ $\Gamma/\mathbf{p}\in\mathcal
T\}$ define mutually inverse order-preserving one-to-one
correspondences between the closed under specilization subsets of
$\Spec \Gamma$ and the hereditary torsion theories in $\Gamma -\Mod$.
\end{theorem}

For our purposes it is convenient to identify for a given module
$M$ the \emph{torsion submodule} $t_Z(M)$  with respect to the
torsion theory $(\mathcal T_Z,\mathcal T_Z^\perp )$.

\begin{proposition} \label{torsion_submodule}
Let $Z\subseteq \Spec \Gamma$ be a closed under specialization subset
and $M$ be a $\Gamma$-module. For an element $x\in M$, the following
assertions are equivalent:

\begin{enumerate}
\item $x$ belongs to $t_Z(M)$ \item $\Ass (\Gamma x)\subseteq Z$ (resp. $\Min \Ass (\Gamma x)\subseteq
Z$)
\item If $\mathbf{p}$ is a prime ideal such
that $\ann_\Gamma (x)\subseteq\mathbf{p}$, then $\mathbf{p}\in Z$
\item There are prime ideals $\mathbf{p}_1,\dots,\mathbf{p}_r\in Z$ (resp. $\mathbf{p}_1,\dots,\mathbf{p}_r\in \Min Z$) and  integers
$n_1,$ $\dots,$ $n_r>0$ such that $\mathbf{p}_1^{n_1}\cdot
\dots\cdot\mathbf{p}_r^{n_r}x=0$
\end{enumerate}
\end{proposition}
\begin{proof}
$1)\Longleftrightarrow 2)\Longleftrightarrow 3)$ Due to the fact
that $\mathcal T_Z$ is closed under taking submodules, assertion 1)
is equivalent to say that $\Gamma x\in\mathcal T_Z$, i.e., to say
that $\Supp (\Gamma x)\subseteq Z$.  But $\Supp (\Gamma x)$ is precisely
the set of prime ideals containing $\ann_\Gamma (x)$ (cf. Proposition
III.4.6 in \cite{Ku}). Moreover,  being $Z$ closed under
specialization, Proposition \ref{minimal_in_Ass_and_Supp } implies
that $\Supp (\Gamma x)\subseteq Z$ holds exactly when $(\Min )\Ass (\Gamma
x)\subseteq Z$.

$3)\Longrightarrow 4)$ Let $\{\mathbf{p}_1,\dots,\mathbf{p}_r\}$ be
the (finite) set of  prime ideals of $\Gamma$ which are minimal
among those containing $\ann_\Gamma (x)$. In particular, they belong
to $Z$. Then we have $\mathbf{p}_1\cdot \ldots
\cdot\mathbf{p}_r\subseteq\mathbf{p}_1\cap\ldots\cap
\mathbf{p}_r=\sqrt{\ann_\Gamma (x)}$, where $\sqrt{I}$ denotes the
radical of $I$, for every ideal $I$ of $\Gamma$.  It follows the
existence of a positive integer $n>0$ such that
$\mathbf{p}_1^n\cdot\ldots
\cdot\mathbf{p}_r^n=(\mathbf{p}_1\cdot\ldots
\cdot\mathbf{p}_r)^n\subseteq \ann_\Gamma (x)$. By Proposition
\ref{minimal_in_Ass_and_Supp }(1), replacing each $\mathbf{p}_i$ by
minimal element of $Z$ contained in it if necessary, we can find the
needed $\mathbf{p}_i$ in $\Min Z$.

$4)\Longrightarrow 3)$ Let $\mathbf{p}_1,\dots,\mathbf{p}_r\in Z$
and $n_1,\dots,n_r>0$ be as in condition 4). Then we have
$\mathbf{p}_1^{n_1}\cdot\ldots \cdot\mathbf{p}_r^{n_r}\subseteq
\ann_\Gamma (x)$. If $\mathbf{p}$ is a prime ideal such that
$\ann_\Gamma (x)\subseteq\mathbf{p}$ then there is some $j=1,\dots,r$
such that $\mathbf{p}_j\subseteq\mathbf{p}$. It follows that
$\mathbf{p}\in Z$ since $Z$ is closed under specialization.
\end{proof}

The following example of closed under specialization subsets of
$\Spec \Gamma$ will be the most interesting for us.

\begin{example}
One defines a transfinite ascending chain of subsets $(Z_i
)_{i\text{ ordinal}}$  as follows. We put $Z_0=\Specm \Gamma$. If $i
>0$ is any ordinal and $Z_j$ has been defined for all $j
<i$, then $Z_i =\bigcup_{j <i}Z_j$, in case $i$ is a limit ordinal,
and $Z_i =Z_{i -1}\cup \Max (\Spec\Gamma\setminus Z_{i -1})$ in case
$i$ is nonlimit. It is not difficult to see that there is a minimal
ordinal $\delta$ such that $\Spec\Gamma =Z_\delta$ and that all
$Z_i$ are closed under specialization. In particular, for each
$\mathbf{p}\in \Spec\Gamma$, there is a minimal ordinal
$i_\mathbf{p}$ such that $\mathbf{p}\in Z_{i_\mathbf{p}}$. This
ordinal is nonlimit and we put $\cht(\mathbf{p})=i_\mathbf{p}$ and
call it the \emph{coheight} of $\mathbf{p}$.
\end{example}

Using Theorem \ref{Gabriel_bijection}, we get a corresponding
transfinite ascending chain of torsion classes $\mathcal
T_0\subseteq\mathcal T_1\subseteq
\ldots\subseteq\mathcal{T}_i\subseteq\ldots$ such that $\Gamma -\Mod
=\mathcal{T}_\delta =\bigcup_{i\leq\delta}\mathcal{T}_i$. Then, for
every $\Gamma$-module $M$, there is uniquely determined (not
necessarily nonlimit) ordinal $i$ such that $M\in\mathcal T_i$ and
$M\not\in\mathcal T_{j}$, for all $j <i$. We also have $t_i
(M)\subseteq t_{j}(M)$, for all $i\leq j$, where $t_i$ denotes the
torsion radical associated to $\mathcal{T}_{i}$.

\begin{corollary} \label{T_i-torsion T_i-1-torsionfree}
Let $\Gamma$ be a commutative Noetherian ring,  $M$ be a nonzero
$\Gamma$-module and $i$ be a nonlimit ordinal. The following
assertions are equivalent:

\begin{enumerate}
\item $t_i (M)=M$ but $t_{i -1}(M)=0$  \item The next two conditions
hold:

\begin{enumerate}
\item For every $x\in M$ there are prime ideals $\mathbf{p}_1,\dots,\mathbf{p}_r$
of coheight exactly $i$ and positive integers $n_1,\dots,n_r>0$ such
that $\mathbf{p}_1^{n_1}\cdot\ldots\cdot\mathbf{p}_r^{n_r}x=0$
\item If $\mathbf{p}$ is a prime ideal of coheight $<i$ and $x\in M$
is an element such that $\mathbf{p} x=0$, then $x=0$.
\end{enumerate}
\item The prime ideals in $\Ass (M)$ have coheight exactly $i$.
\end{enumerate}
\end{corollary}
\begin{proof}
$1)\Longleftrightarrow 3)$ By Proposition \ref{torsion_submodule}
and the fact that the mentioned torsion theories are hereditary, we
have that $t_i (M)=M$ iff $\Ass (M)\subseteq Z_i$ and $t_{i -1}(M)=0$
iff $\Ass (M)\cap Z_{i -1}=\emptyset$. Therefore assertion 1 holds if
and only if $\Ass (M)\subseteq Z_i\setminus Z_{i -1}$, which is
equivalent to assertion 3.

$2)\Longrightarrow 1)$ From Proposition \ref{torsion_submodule} and
condition 2.a we get that $t_i (M)=M$. On the other hand, if we had
$0\neq x\in t_{i -1}(M)$ that same proposition would give that
$\emptyset\neq \Ass (\Gamma x)\subseteq Z_{i -1}$. We then get
$g\in\Gamma$ such that $gx\neq 0$ and $\ann_\Gamma (gx)=\mathbf{p}$
is a prime ideal in $Z_{i -1}$. That would contradict condition
2.b).

$1),3)\Longrightarrow 2)$ Let's prove condition 2.b by way of
contradiction. Suppose  that there are $0\neq x\in M$ and
$\mathbf{p}\in Z_{i -1}$ such that $\mathbf{p}x=0$. Taking a maximal
element in the set $\{\ann_\Gamma (gx):$ $g\in G\text{ and }gx\neq
0\}$, we obtain a $\mathbf{q}\in \Ass (\Gamma x)\subseteq \Ass (M)$ (cf.
\cite{mat}[Teorem 6.1]) such that $\mathbf{p}\subseteq\mathbf{q}$.
Since $Z_{i -1}$ is closed under specialization we get that
$\mathbf{q}\in Z_{i -1}$, against assertion 3).

We next prove condition 2.a. Let us take $0\neq x\in M$. Then, by
Proposition \ref{torsion_submodule}, we have prime ideals
$\mathbf{p}_1,\ldots, \mathbf{p}_r\in Z_i$ (hence of coheight
$\leq i$) and positive integers $n_1, \ldots, n_r>0$ such that
$\mathbf{p}_1^{n_1}\cdot \ldots\cdot\mathbf{p}_r^{n_r}x=0$. It is
not restrictive to choose the $\mathbf{p}_i$ and  the $n_i$ in
such a way that the latter ones are minimal, i.e., that
$\mathbf{p}_1^{n_1}\cdot \ldots\cdot\mathbf{p}_k^{n_k-1}\cdot
\ldots\cdot\mathbf{p}_r^{n_r}x\neq 0$ for all $k=1,\ldots, r$.
That immediately implies the existence of elements $g_k\in\Gamma$
such that $g_kx\neq 0$ and $\mathbf{p}_k\subseteq \ann_\Gamma
(g_kx)$, for all $k=1,\ldots,r$. By \cite{mat}[Theorem 6.1], we
find $\mathbf{q}_k\in \Ass (\Gamma x)\subseteq \Ass (M)$ such that
$\mathbf{p}_k\subseteq\mathbf{q}_k$, for all $k=1,\ldots, r$. But
then, by assertion 3), we have $i =\cht(\mathbf{q}_k)\leq
\cht(\mathbf{p}_k)\leq i$ for $k=1,\ldots, r$. Therefore we have
$\cht(\mathbf{p}_k)=i$, for $k=1,\ldots, n$.
\end{proof}

Our next goal is to give the precise structure of the
$\Gamma$-modules in $\mathcal{T}_0$, which is actually given by a
more general result, Proposition \ref{coprimes_in_Ass} below, which
will follow from the following  strengthened version of the chinese
reminder's theorem:

\begin{lemma} \label{chinese_reminder}
Let $I_1,\ldots, I_r$ ($r>1$) be pairwise distinct ideals of
$\Gamma$. The following assertions are equivalent:

\begin{enumerate}
\item $I_i$ and $I_j$ are coprime, for all $i\neq j$ \item The
canonical ring homomorphism $\Gamma\longrightarrow\prod_{1\leq i\leq
r}\Gamma/I_i$ is surjective.
\end{enumerate}
In such case $\bigcap_{1\leq i\leq j}I_i=I_1\cdot \ldots\cdot
I_r$.
\end{lemma}

\begin{proof}
See \cite{am}, Proposition 1.10, i).
\end{proof}

In the rest of the paper, if $\mathbf{p}\in \Spec\Gamma$ and $M$ is
a $\Gamma$-module, we shall denote by $M(\mathbf{p})$ the submodule
consisting of those $x\in M$ such that $\mathbf{p}^nx=0$, for some
$n\geq 0$. Note that, in such case, if $\mathbf{p}\in
Ass(M(\mathbf{p}))$ then $MinAss(M(\mathbf{p}))=\{p\}$.

\begin{proposition} \label{coprimes_in_Ass}
Let $M$ be a $\Gamma$ module such that $\Min \Ass (M)$ consists of
pairwise coprime ideals (e.g. if $\Ass (M)\subseteq \Specm \Gamma$).
Then $\Min \Ass (M)=\Ass (M)$ and $M=\oplus_{\mathbf{p}\in
\Ass (M)}M(\mathbf{p})$.
\end{proposition}
\begin{proof}
We shall prove that $M=\oplus_{\mathbf{p}\in
\Min \Ass (M)}M(\mathbf{p})$. It will follow that
$\Ass (M)=\bigcup_{\mathbf{p}\in
\Min \Ass (M)}\Ass (M(\mathbf{p}))=\bigcup_{\mathbf{p}\in
\Min \Ass (M)}\{\mathbf{p}\}=\Min \Ass (M)$ and the result will follow.

Let us fix $\mathbf{p}\in \Min \Ass (M)$ and take $$x\in
M(\mathbf{p})\cap(\oplus_{\mathbf{q}\in \Min \Ass
(M),\mathbf{q}\neq\mathbf{p}}M(\mathbf{q})).$$ Then we have
inclusions

\begin{center}
$\Ass (\Gamma x)\subseteq \Ass (M(\mathbf{p}))\cap
\Ass (\oplus_{\mathbf{q}\in
\Min \Ass (M),\mathbf{q}\neq\mathbf{p}}M(\mathbf{q}))\subseteq\{\mathbf{p}\}\cap
(\Min \Ass (M)\setminus\{\mathbf{p}\})=\emptyset$.
\end{center}
It follows that $x=0$ and, hence, the sum of the $M(\mathbf{q})$,
with $\mathbf{q}\in \Min \Ass (M)$, is direct.

Let us consider now $Z:=\Supp (M)$, which is a subset of
$\Spec\Gamma$ closed under specialization. Then, by Theorem
\ref{Gabriel_bijection}, $M$ belongs to $\mathcal{T}_Z$ and hence
$t_Z(M)=M$. If now $x\in M$ then Proposition
\ref{torsion_submodule} guarantees the existence of distinct prime
ideals $\mathbf{p}_1,\ldots, \mathbf{p}_r\in \Min \Supp (M)$ and
positive integer $n_1,\ldots, n_r>0$ such that
$\mathbf{p}_1^{n_1}\cdot \ldots\cdot\mathbf{p}_r^{n_r}x=0$. The
$\mathbf{p}_i$ are pairwise coprime since $\Min \Supp (M)=\Min
\Ass (M)$ (see Proposition \ref{minimal_in_Ass_and_Supp }). But
then it follows easily that the ideals $\mathbf{p}_i^{n_i}$ are
also pairwise coprime. Then $\Gamma x$ is a module over the factor
ring $\Gamma/\mathbf{p}_1^{n_1}\cdot \ldots\mathbf{p}_r^{n_r}$.
But, by Lemma \ref{chinese_reminder}, we know that
$\mathbf{p}_1^{n_1}\cdot
\ldots\cdot\mathbf{p}_r^{n_r}=\bigcap_{1\leq i\leq
r}\mathbf{p}_i^{n_i}$, and then the canonical map

\begin{center}
$\Gamma/\mathbf{p}_1^{n_1}\cdot
\ldots\cdot\mathbf{p}_r^{n_r}\longrightarrow\prod_{1\leq i\leq
r}\Gamma/\mathbf{p}_i^{n_i}$
\end{center}
is a ring isomorphism. It follows that in the ring
$\Gamma/\mathbf{p}_1^{n_1}\cdot \ldots\cdot\mathbf{p}_r^{n_r}$ we
can decompose $\bar{1}=\bar{g}_1+\ldots+\bar{g}_r$, where
$g_i\in\mathbf{p}_1^{n_1}\cdot
\ldots\mathbf{p}_{i-1}^{n_{i-1}}\cdot\mathbf{p}_{i+1}^{n_{i+1}}\ldots\cdot\mathbf{p}_r^{n_r}$.
Then $x=\sum_{1\leq i\leq r}g_ix$ and $\mathbf{p_i}^{n_i}g_ix=0$,
for $i=1,\ldots, r$. It follows that $x\in\oplus_{\mathbf{p}\in
\Min \Ass (M)}M(\mathbf{p})$, and we get the desired equality
$M=\oplus_{\mathbf{p}\in \Min \Ass (M)}M(\mathbf{p})$.
\end{proof}

\begin{proposition} \label{coprime_assassins}
Let $M$ and $N$ be $\Gamma$-modules such that $\mathbf{p}$ and
$\mathbf{q}$ are coprime whenever $\mathbf{p}\in \Ass (M)$ and
$\mathbf{q}\in \Ass (N)$ (resp. $\mathbf{p}\in \Min \Ass (M)$ and
$\mathbf{q}\in \Min \Ass (N)$). The equality

\begin{center}
$\Ext_\Gamma^i(M,N)=0=\Ext_\Gamma^i(N,M)$
\end{center}
holds for all $i\geq 0$.
\end{proposition}
\begin{proof}
Since we have $\Min \Ass (M)=\Min \Supp (M)$ and similarly for $N$
it follows that $\mathbf{p}$ and $\mathbf{q}$ are coprime whenever
$\mathbf{p}\in \Supp (M)$ and $\mathbf{q}\in \Supp (N)$. If

\begin{center}
$0\rightarrow M\longrightarrow I^0\longrightarrow
I^1\longrightarrow \ldots$
\end{center}
is the minimal injective resolution of $M$ in $\Gamma -\Mod$ and
$E(\Gamma /\mathbf{p})$ is an injective indecomposable
$\Gamma$-module appearing as direct summand of some $I^i$, then
$\mathbf{p}\in \Supp  (M)$ (cf. \cite{mat}[Theorem 18.7]). It follows
that $\Hom_\Gamma (N,I^i)=0$, and hence $\Ext_\Gamma ^i(N,M)=0$, for
all $i\geq 0$. That $\Ext_\Gamma ^i(M,N)=0$ for all $i\geq 0$ follows
by symmetry.
\end{proof}

%

\section{Algebras with a commutative Harish-Chandra subalgebra and lifting of torsion theories}
Throughout the rest of the paper  $U$ and $\Gamma$ satisfy the
Setup~\ref{setup}. We denote by $j:\Gamma\hookrightarrow U$ the
canonical inclusion and by $j_*:U-\Mod\longrightarrow\Gamma-\Mod$
the restriction of scalar functor. It is clear that if $\mathcal{T}$
is a (hereditary) torsion class in $\Gamma -\Mod$, then
$\hat{\mathcal{T}}=j_*^{-1}(\mathcal{T}):=\{T\in U-\Mod:$
$j_*(T)\in\mathcal{T}\}$ is a (hereditary) torsion class in
$U-\Mod$. However, if $M$ is an $U$-module, then its torsion
$\Gamma$-submodule $t(M)$ and its torsion $U$-submodule $\hat{t}(M)$
satisfy an inclusion $\hat{t}(M)\subseteq t(M)$ that might be
strict. Equality happens exactly when $t(M)$ is an $U$-submodule of
$M$. That justifies the following.

\begin{definition} \label{definition_liftable_torsion_theory}
A torsion theory $(\mathcal{T},\mathcal{F})$ in $\Gamma -\Mod$ is
called \emph{liftable} to $U-\Mod$ in case $t(M)$ is a $U$-submodule
of $M$, for every $U$-module $M$.
\end{definition}

The following is a general criterion for the lifting of a torsion
theory.

\begin{proposition} \label{liftability_-_general_criterion}
Let $Z\subseteq \Spec\Gamma$ be a closed under specialization subset
and $(\mathcal{T}_Z,\mathcal{F}_Z)$ be its associated torsion theory
in $\Gamma -\Mod$. The following assertions are equivalent:

\begin{enumerate}
\item $(\mathcal{T}_Z,\mathcal{F}_Z)$ is liftable to $U-\Mod$ \item
For each prime ideal $\mathbf{p}$ (minimal) in $Z$, the  $U$-module
$U/U\mathbf{p}$ belongs to $\mathcal{T}_Z$ when looked at as
$\Gamma$-module.
\end{enumerate}
\end{proposition}
\begin{proof}
$(1)\Longrightarrow (2)$ Let us take $\mathbf{p}\in Z$. Then the
canonical generator $x=1+U\mathbf{p}$ of $U/U\mathbf{p}$ belongs
to $t_Z(U/U\mathbf{p})$ (see Proposition \ref{torsion_submodule}).
Since $t_Z(U/U\mathbf{p})$ is a $U$-submodule of $U/U\mathbf{p}$
we conclude that $U/U\mathbf{p} =t_Z(U/\mathbf{p})$ and condition
(2) holds.

$(2)\Longrightarrow (1)$ Let $M\neq 0$ be an arbitrary nonzero
$U$-module. If $0\neq x\in t_Z(M)$ then, by Proposition
\ref{torsion_submodule}, there are
$\mathbf{p}_1,\dots,\mathbf{p}_r\in \Min Z$ and positive integers
$n_1,\dots,n_r>0$ such that
$\mathbf{p}_1^{n_1}\cdot\ldots\cdot\mathbf{p}_r^{n_r}x=0$.   We
shall prove that $Ux\subseteq t_Z(M)$ by induction on
$k=n_1+\ldots+n_r$. If $k=1$ then we have a minimal $\mathbf{p}\in
Z$ such that $\mathbf{p} x=0$. Then we get an epimorphism of
$U$-modules $U/U\mathbf{p}\twoheadrightarrow Ux$
($\bar{u}=u+U\mathbf{p}\rightsquigarrow ux$) whose domain belongs
to $\mathcal T_Z$ when viewed as a $\Gamma$-module. Then $Ux$
belongs to $\mathcal{T}_Z$ when viewed as a $\Gamma$-module, so
that $Ux\subseteq t_Z(M)$.

 Suppose  now that $k>1$. If $\mathbf{p}_rx=0$ then we are done. So we
can assume that $\mathbf{p}_rx\neq 0$. The induction hypothesis says
that $U\mathbf{p}_rx\subseteq t_Z(M)$, from which it follows that
the assignment $\bar{u}=u+U\mathbf{p}_r\rightsquigarrow
\overline{ux}=ux+t_Z(M)$ gives a well-defined map
$f:U/U\mathbf{p}_r\longrightarrow M/t_Z(M)$, which is clearly a
homomorphism of $\Gamma$-modules. Then we have that
$\Im(f)=(Ux+t_Z(M))/t_Z(M)\in\mathcal T_Z$ since $U/U\mathbf{p}_r$
belongs to $\mathcal T_Z$. But we also have that
$\Im(f)\in\mathcal{F}_Z$ because $\Im(f)$ is a $\Gamma$-submodule of
$M/t_Z(M)$. It follows that $\Im(f)=0$, so that $Ux\subseteq
t_Z(M)$.
\end{proof}

Note that in our setting the commutative algebra $\Gamma$ always
has finite Krull dimension, so that the (co)height of any of its
prime ideal is a natural number. We are now in the position to
prove our main result, which implies Theorem~A.

\begin{theorem} \label{liftability_of_canonical_torsion_theories}
\begin{itemize}
\item[(1)] Let $i$ be any natural number. The torsion theory
$(\mathcal{T}_i,\mathcal{F}_i)$ is liftable to $U-\Mod$.
\item[(2)]\label{same_coheight_in_Ass(M)} Let $M$ be a simple
$U$-module. There exists a (unique) natural number $i$ such that
$t_i(M)=M$ and $t_{i-1}(M)=0$. In that case, all prime ideals in
$\Ass (M)$ have coheight exactly $i$.
\end{itemize}
\end{theorem}
\begin{proof}
We prove the first statement by induction on $i$. If $i=0$ we take
$\mathbf{m}\in \Min Z_0=Z_0=\Specm \Gamma$. In order to prove that
$U/U\mathbf{m}\in\mathcal{T}_0$, thus  ending the proof (cf.
Proposition\ref{liftability_-_general_criterion}),  it is enough
to prove that $\frac{\Gamma u\Gamma
+U\mathbf{m}}{U\mathbf{m}}\cong\frac{\Gamma u\Gamma}{\Gamma
u\Gamma\cap U\mathbf{m}}$ is a 'left' $\Gamma$-module in
$\mathcal{T}_0$, for all $u\in U$. Indeed we have an epimorphism
in $\Gamma -\Mod$

\begin{center}
$\frac{\Gamma u\Gamma}{\Gamma
u\mathbf{m}}\twoheadrightarrow\frac{\Gamma u\Gamma}{\Gamma
u\Gamma\cap U\mathbf{m}}$.
\end{center}
But since $\Gamma u\Gamma$ is  finitely generated as right
$\Gamma$-module it follows that $\frac{\Gamma u\Gamma}{\Gamma
u\mathbf{m}}$ is finite dimensional as $K$-vector space. In
particular $\frac{\Gamma u\Gamma}{\Gamma u\Gamma\cap U\mathbf{m}}$
is a 'left' $\Gamma$-module of finite length and hence belongs to
$\mathcal{T}_0$.

 Suppose  now that $i>0$ and $i<d=Kdim(\Gamma )$ (the case $i\geq d$
is trivial). If $\mathbf{p}\in \Min Z_i$ and $\cht(\mathbf{p})<i$  then
the induction hypothesis says that
$U/U\mathbf{p}\in\mathcal{T}_{i-1}\subset\mathcal{T}_i$. We assume
then that $\cht(\mathbf{p})=i$.   According to Proposition
\ref{liftability_-_general_criterion}, it will be enough to prove
that $U/U\mathbf{p}$ belongs to $\mathcal{T}_i$ when viewed as a
$\Gamma$-module. This is turn equivalent to prove that, for each
$u\in U$, all the prime ideals of $\Gamma$ containing $\ann_\Gamma
(u+U\mathbf{p})=(U\mathbf{p}:u):=\{g\in\Gamma :$ $gu\in
U\mathbf{p}\}$ have coheight $\leq i$ (cf. Proposition \ref{torsion_submodule}). Therefore our goal is to prove that the Krull dimension
of the algebra $\Gamma/(U\mathbf{p}:u)$ is $\leq i$, for all $u\in
U$. For that we shall use  the fact that the Krull dimension of this
latter algebra coincides with its Gelfand-Kirillov dimension (cf
\cite{KL}[Proposition 7.9])

We fix an element $u\in U$,  a finite set of generators
$\{u=u_1,u_2,\ldots, u_n\}$ of $\Gamma u\Gamma$ as right
$\Gamma$-module and a finite set of generators $\{t_1,\ldots, t_m\}$
of $\Gamma$ as a $K$-algebra. We consider the filtration
$(F_k)_{k\geq 0}$ on $\Gamma$ obtained by taking as $F_k$ the vector
subspace of $\Gamma$ generated by the monomials of degree $\leq k$
on the $t_i$. The induced filtration on $\Gamma /(U\mathbf{p}:u)$ is
given by $(\frac{F_k+(U\mathbf{p}:u)}{(U\mathbf{p}:u)})_{k\geq 0}$.
The multiplication map $\bar{g}\rightsquigarrow gu+U\mathbf{p}$ is a
$K$-linear isomorphism
$\frac{F_k+(U\mathbf{p}:u)}{(U\mathbf{p}:u)}\stackrel{\cong}{\longrightarrow}\frac{F_ku+U\mathbf{p}}{U\mathbf{p}}$,
for each $k\geq 0$.

Due to our choices, we have that $t_iu_j=\sum_{1\leq l\leq
n}u_lg_{ij}^l$, with $g_{ij}^l\in\Gamma$,  for all $i=1,\ldots, r$
and $j=1,\ldots ,n$. There exists a minimal positive integer $s>0$
such that $\{g_{ij}^l\}\subset F_s$. An easy induction gives that
$F_ku_j\subseteq\sum_{1\leq i\leq n}u_iF_{sk}$, for all $k\geq 0$
and all $j=1,\ldots,n$. In particular we have
$F_ku\subseteq\sum_{1\leq i\leq n}u_iF_{sk}$, and hence
$\frac{F_ku +U\mathbf{p}}{U\mathbf{p}}\subseteq\sum_{1\leq i\leq
n}\frac{u_iF_{ks}+U\mathbf{p}}{U\mathbf{p}}$, for all $k\geq 0$.
Note that we have a surjective $K$-linear map

\begin{center}
$\frac{F_{sk}+U\mathbf{p}}{U\mathbf{p}}\twoheadrightarrow\frac{u_iF_{ks}+U\mathbf{p}}{U\mathbf{p}}$
($g+U\mathbf{p}\rightsquigarrow u_ig+U\mathbf{p}$).
\end{center}
Then,  taking $K$-dimensions, we obtain

\begin{center}
$dim(\frac{F_ku+U\mathbf{p}}{U\mathbf{p}})\leq s\cdot
dim(\frac{F_{ks}+U\mathbf{p}}{U\mathbf{p}})$,
\end{center}
and hence
\begin{center}
$\frac{log(dim(\frac{F_ku +U\mathbf{p}}{U\mathbf{p}}))}{log(k)}\leq
\frac{log(s\cdot dim(\frac{F_{ks}
+U\mathbf{p}}{U\mathbf{p}}))}{log(k)}$,  \hspace*{1cm}(*)
\end{center}
for all $k>0$. Note that we obtain a filtration $(F'_k)_{k\geq 0}$
of the algebra $\Gamma$ by putting $F'_k=F_{sk}$, for all $k\geq 0$.
Then, by applying limit superior to the inequality (*) and bearing
in mind that the Gelfand-Kirillov dimension decreases by passing to
factor algebras, we get that

\begin{center}
$GKdim(\Gamma /(U\mathbf{p}:u))\leq
GKdim(\Gamma/(U\mathbf{p}\cap\Gamma))\leq GKdim(\Gamma
/\mathbf{p})=i$.
\end{center}

This proves the first statement of the theorem. Let us now put
$i=min\{j\geq 0:$ $M\in\mathcal{T}_j\}$. Then we have $t_i(M)=M$ and
$t_{i-1}(M)\subsetneq M$ (convening that $t_{-1}(M)=0$). By (1), it
follows that $t_{i-1}(M)$ is a proper $U$-submodule of $M$. The
simplicity of $M$ gives that $t_{i-1}(M)=0$ and, using Corollary
\ref{T_i-torsion T_i-1-torsionfree},   the proof is completed.
\end{proof}

\begin{question and remark}\label{rem-question}
According to Proposition~\ref{coprimes_in_Ass}, if $M$ is a simple
$U$-module and the prime ideals in $\Ass (M)$ are pairwise coprime
(e.g. if $M\in\mathcal{T}_0$) then, as $\Gamma$-module, we have a
decomposition $M=\oplus_{\mathbf{p}\in \Ass (M)}M(\mathbf{p})$. For
an arbitrary simple $M$, using
Theorem~\ref{liftability_of_canonical_torsion_theories}, it is not
difficult to see that the sum $\sum_{\mathbf{p}\in \Ass
(M)}M(\mathbf{p})$ is direct, so that $\oplus_{\mathbf{p}\in \Ass
(M)}M(\mathbf{p})$ is a $\Gamma$-submodule of $M$. Is it a
$U$-submodule (so that the equality $M=\oplus_{\mathbf{p}\in \Ass
(M)}M(\mathbf{p})$ holds)?
\end{question and remark}

Given a simple $U$-module, one needs recipes to calculate the
$i\geq 0$ such that $t_i(M)=M$ and $t_{i-1}(M)=0$. Recall that a
subset $\{g_1,\ldots, g_r\}\subset\Gamma$ is called a
\emph{regular sequence} in case $\sum_{1\leq i\leq n}\Gamma
g_i\neq\Gamma$ and $\bar{g}_k:=g_k+\sum_{1\leq i<k}\Gamma g_i$ is
not a zero divisor in $\Gamma/\sum_{1\leq i<k}\Gamma g_i$,  for
all $k=1,\ldots, n$. In that case $r$ is called the \emph{length}
of the regular sequence. We refer the reader to \cite{mat}[pages
136 and 250] for the definitions of Cohen-Macaulay and
equidimensional commutative rings, that we use in the following
result.

\begin{proposition} \label{regular_sequence}
 Suppose  that  $\Gamma$ is Cohen-Macaulay and equidimensional  and
let $d=Kdim(\Gamma )$ be its Krull dimension. If $M$ is a $U$-module
such that all ideals in $\Ass (M)$ have the same coheight (e.g. a
simple $U$-module),
 then the
following assertions are equivalent:

\begin{enumerate}
\item $t_i(M)=M$ and $t_{i-1}(M)=0$ \item There is a regular
sequence in $\Gamma$,  maximal with the property of annihilating
some $x\in M\setminus\{0\}$, which has length $d-i$.
\end{enumerate}
\end{proposition}
\begin{proof}
The equidimensionality guarantees that
$\hit(\mathbf{p})+\cht(\mathbf{p})=d$, for all $\mathbf{p}\in
\Spec\Gamma$ (cf. \cite{Ku}[Corollary II.3.6]). Note also that if
$\{g_1,\ldots, g_k\}$ is a regular sequence contained in
$\ann_\Gamma (x)$, for some $x\in M\setminus\{0\}$, then,
replacing if necessary $x$ by some $gx\neq 0$ with $g\in G$, it is
not restrictive to assume that $\ann_\Gamma (x)=\mathbf{q}$, for
some prime ideal $\mathbf{q}\in \Ass (M)$. So assertion (2) is
equivalent to the following:

(2') There is a regular sequence in $\Gamma$ of length $d-i$
contained in some $\mathbf{q}\in \Ass (M)$ and maximal with that
property.

By \cite{Ku}[Theorem VI.3.14] and the fact that all prime ideals in
$\Ass (M)$ have the same (co)height, this condition 2' is in turn
equivalent to say that $d-i=\hit(\mathbf{q})$, for every
$\mathbf{q}\in \Ass (M)$. Therefore assertion 2) holds if, and only
if, $\cht(\mathbf{q})=i$ for all $\mathbf{q}\in \Ass (M)$. By
Corollary \ref{T_i-torsion T_i-1-torsionfree},
 this is equivalent to assertion
(1).

\end{proof}

\section{An approximation to the assassin of a  $U$-module}
The preceding section  shows that, given a simple $U$-module,  its
assassin as $\Gamma$-module, $\Ass (M)$, is an important
invariant. Therefore it is natural to give recipes to approximate
this subset of $\Spec(\Gamma)$. We will see in this section that,
knowing a prime $\mathbf{p}\in \Ass (M)$ and the finite subset
$\{u_1,\ldots, u_n\}\subset U$ of our setup (see \ref{setup}), one
can give a precise subset of $\Spec\Gamma$ in which $\Ass (M)$ is
contained.

We will follow the terminology used for maximal ideals in
\cite{dfo:hc}
 and, given $u\in U$, we denote by $X_u$ the set of
pairs $(\mathbf{q},\mathbf{p})\in \Spec\Gamma\times \Spec\Gamma$ such
that $\frac{\Gamma u\Gamma}{\mathbf{q}u\Gamma+\Gamma
u\mathbf{p}}\neq 0$ (or equivalently
$\frac{\Gamma}{\mathbf{q}}\otimes_\Gamma\Gamma
u\Gamma\otimes_\Gamma\frac{\Gamma}{\mathbf{p}}\neq 0$). For
simplicity, we shall write $\mathbf{q}\equiv_u\mathbf{p}$ whenever
$(\mathbf{q},\mathbf{p})\in X_u$.

Note that, due to Nakayama lemma,  if $H$ if a finitely generated
$\Gamma$-module and $\mathbf{q}\in \Supp (H)$ then $\mathbf{q}H\neq
H$. We will use this fact in the proof of the following result,
which is a crucial tool for our purposes.

\begin{lemma} \label{crucial_lemma_on_assassins}
Let $M$ be a $U$-module. The following assertions hold:

\begin{enumerate}
\item If $u\in U$, $x\in M$ and $\mathbf{q}\in \Supp (\Gamma ux)$, then
there exists $\mathbf{p}\in \Ass (\Gamma x)$ such that
$\mathbf{q}\equiv_u\mathbf{p}$ \item If all prime ideals in $\Ass (M)$
have the same coheight, then there is an inclusion

\begin{center}
$\Ass (\Gamma (x+y))\subseteq \Ass (\Gamma x)\cup \Ass (\Gamma y)$,
\end{center}
for all $x,y\in M$.
\end{enumerate}
\end{lemma}
\begin{proof}
1) We have $\mathbf{q}\in \Supp (\Gamma ux)\subseteq \Supp (\Gamma
u\Gamma x)$. It follows that $\frac{\Gamma u\Gamma
x}{\mathbf{q}u\Gamma x}\neq 0$ since our setup \ref{setup}
guarantees that $\Gamma u\Gamma x$ is a finitely generated
$\Gamma$-module. The assignment $\bar{v}\otimes y\rightsquigarrow
\overline{vy}$ gives a surjective $K$-linear map

\begin{center}
$\frac{\Gamma u\Gamma}{\mathbf{q}u\Gamma}\otimes_\Gamma\Gamma
x\twoheadrightarrow\frac{\Gamma u\Gamma x}{\mathbf{q}u\Gamma x}\neq
0$.
\end{center}
It follows that  $\frac{\Gamma
u\Gamma}{\mathbf{q}u\Gamma}\otimes_\Gamma\Gamma x\neq 0$. But
$\Gamma x$ admits a finite filtration with successive factors
isomorphic to $\Gamma /\mathbf{p}$, with $\mathbf{p}\in \Supp (\Gamma
x)$ (see \cite{Ku}[Proposition VI.2.6]). We conclude that there is a
$\mathbf{p}'\in \Supp (\Gamma x)$ such that $\frac{\Gamma
u\Gamma}{\mathbf{q}u\Gamma}\otimes_\Gamma\frac{\Gamma}{\mathbf{p}'}\neq
0$. Choosing now $\mathbf{p}\in \Ass (\Gamma x)$ such that
$\mathbf{p}\subseteq\mathbf{p}'$, we get that $\frac{\Gamma
u\Gamma}{\mathbf{q}u\Gamma}\otimes_\Gamma\frac{\Gamma}{\mathbf{p}}\neq
0$ and hence $\mathbf{q}\equiv_u\mathbf{p}$.

3) Since we have an inclusion $\Gamma (x+y)\subseteq\Gamma x+\Gamma
y$ it will be enough to check that $\Ass (\Gamma x+\Gamma y)\subseteq
\Ass (\Gamma x)\cup \Ass (\Gamma y)$. To do that, we consider the
canonical exact sequence in $\Gamma -\Mod$:

\begin{center}
$0\rightarrow\Gamma x\cap\Gamma y\longrightarrow\Gamma x\oplus\Gamma
y\longrightarrow\Gamma x+\Gamma y\rightarrow 0$,
\end{center}
from which we get that $\Ass (\Gamma x+\Gamma y)\subseteq \Supp (\Gamma
x\oplus\Gamma y)=\Supp (\Gamma x)\cup \Supp (\Gamma y)$.

By hypothesis, all prime ideals in $\Ass (M)$ have the same coheight,
which implies that all of them are minimal in $\Supp (M)$. As a
consequence, if $\mathbf{q}\in \Ass (\Gamma x+\Gamma y)$ and we assume
that $\mathbf{q}\in \Supp (\Gamma x)$, then $\mathbf{q}$ is minimal in
$\Supp (\Gamma x)$. This implies that $\mathbf{q}\in \Min \Supp (\Gamma
x)=\Min \Ass (\Gamma x)\subseteq \Ass (\Gamma x)$. We  replace $x$ by $y$
in case $\mathbf{q}\in \Supp (\Gamma y)$, and the proof is finished.
\end{proof}

\subsection{Proof of Theorem B}
We are now in the position to prove Theorem~B.

 If
 $\mathbf{q}\in \Ass (M)$ then  we have $\mathbf{q}=\ann_\Gamma (ux)$, for some $u\in U$. If $u\in\Gamma$ then $\mathbf{q}=\mathbf{p}$ and there
 is nothing to prove.  So we assume $u\not\in\Gamma$,  in which case
$u$ is a sum of products of the form

\begin{center}
$g_1u_{k_1}g_2\ldots g_ru_{k_r}g_{r+1}$,
\end{center}
where the $g_k$ belong to $\Gamma$ and the $k_1,\ldots, k_r$
belong to $\{1,\ldots, n\}$. Lemma
\ref{crucial_lemma_on_assassins} allows us to assume, without loss
of generality, that

\begin{center}
$u=g_1u_{k_1}g_2\ldots g_ru_{k_r}g_{r+1}$,
\end{center}
something that we do from now on in this proof.

We then have

\begin{center}
$\mathbf{q}\in \Ass (\Gamma ux)\subseteq \Ass (\Gamma
u_{k_1}g_2...g_ru_{k_r}g_{r+1}x)$.
\end{center}
By Lemma \ref{crucial_lemma_on_assassins}(1), there is a
$\mathbf{q}_1\in \Ass (\Gamma g_2u_{k_2}...g_ru_{k_r}g_{r+1}x)$
such that $\mathbf{q}\equiv_{u_{k_1}}\mathbf{q}_1$. By induction
we get a sequence
$\mathbf{q}=\mathbf{q}_0,\mathbf{q}_1,...,\mathbf{q}_r$ of prime
ideals in $\Ass (M)$, whence of coheight exactly
$\cht(\mathbf{p})$ (see
Theorem~\ref{liftability_of_canonical_torsion_theories},
 such that $\mathbf{q}_r\in \Ass
(\Gamma g_{r+1}x)$ and
$\mathbf{q}_{i-1}\equiv_{u_{k_i}}\mathbf{q}_i$ for $i=1,...,r$.
But $\Ass (\Gamma g_{r+1}x)=\{\mathbf{p}\}$ since $\ann_\Gamma
(x)=\mathbf{p}$ is a prime ideal and $g_{r+1}x\neq 0$. Then
$\mathbf{q}_r=\mathbf{p}$ and the proof is finished.

\vspace*{0.3cm}

Theorem~B suggests to define, for each $0\leq i\leq d$,  a (not
necessarily symmetric) relation $\equiv$ in the set $\Min Z_i$ of
prime ideals of coheight $i$ by saying that
$\mathbf{q}\equiv\mathbf{p}$ if, and only if,  there are a
sequence $\mathbf{q}=\mathbf{q}_0,\mathbf{q}_1,\ldots,
\mathbf{q}_s=\mathbf{p}$ in $\Min Z_i$ and a sequence of indices
$k_1,\ldots, k_s\in\{1,\ldots, n\}$ such that
$\mathbf{q}_{i-1}\equiv_{u_{k_i}}\mathbf{q}_i$, for all
$i=1,\ldots, s$.

\begin{corollary} \label{same_coheight_in_Ass(M)_for_simple}
If $M$ is a simple $U$-module and $\mathbf{p},\mathbf{q}\in \Ass (M)$
then $\mathbf{q}\equiv\mathbf{p}$.
\end{corollary}
\begin{proof}
As $U$-module, $M$ is generated by any of its nonzero elements.
Choose $0\neq x\in M$ such that $\ann_\Gamma (x)=\mathbf{p}$ and
apply Theorem~\ref{liftability_of_canonical_torsion_theories}.
\end{proof}

We obtain immediately the following refinement of
Proposition~\ref{coprimes_in_Ass}.

\begin{corollary}
Let $M$ be a simple $U$-module and take $\mathbf{p}\in \Ass (M)$, with
$\cht(\mathbf{p})=i$. Suppose  that $\mathbf{q}$ and $\mathbf{q}'$ are
coprime whenever  $\mathbf{q}\neq\mathbf{q}'$ are distinct prime
ideals of $\Gamma$ of coheight $i$  such that
$\mathbf{q}\equiv\mathbf{p}$ and $\mathbf{q}'\equiv\mathbf{p}$. Then
we have a decomposition  $M=\oplus_{\mathbf{q}\in
\Ass (M)}M(\mathbf{q})$ as $\Gamma$-module.
\end{corollary}
\begin{proof}
By Theorem~B, we have an inclusion $\Ass
(M)\subseteq\{\mathbf{q}\in \Spec\Gamma :$
$\cht(\mathbf{q})=i\text{ and }\mathbf{q}\equiv\mathbf{p}\}$.
Therefore the elements of $\Ass (M)$ are pairwise coprime and
Proposition \ref{coprimes_in_Ass} applies.
\end{proof}

The following example shows that in some circumstances (usually
when the coheight is large), Theorem~B is not sufficient to
approximate Ass(M).

\begin{example}\label{first Weyl}
Let $U=A_n(K)$ be the  Weyl algebra given by generators
$X_1,\ldots, X_n,Y_1,\ldots,Y_n$ subject to the relations

\begin{center}
$X_iX_j-X_jX_i=0=Y_iY_j-Y_jY_i$

$X_iY_j-Y_jX_i=\delta_{ij}$,
\end{center}
for all $i,j\in\{1,\ldots,n\}$, where $\delta_{ij}$ is the Kronecker
symbol. Assume $n>1$, put  $t_i=X_iY_i$  and put
$\Gamma=K[t_1,\ldots, t_n]$. Then $\Gamma$ and $U$ satisfy the
conditions of our setup \ref{setup} by taking  $u_j\in \{X_{\sigma
(j)}, Y_\sigma (j)\}$ for all $j=1,\ldots, n$, where $\sigma\in S_n$
is any permutation. If $\mathbf{p}=\Gamma (t_1-1)$ then
$\mathbf{q}\equiv\mathbf{p}$, for every prime ideal $\mathbf{q}\in
\Spec(\Gamma )$ of height $1$.
\end{example}
\begin{proof}
 For simplicity put $u_i=Y_i$ ($i=1,...,n$), the other choices being treated similarly. Then one
readily shows the equalities

\begin{center}
$Y_it_j=t_jY_i$ ($i\neq j$)

$Y_it_i=(t_i-1)Y_i$ (equivalently $t_iY_i=Y_i(t_i+1)$),
\end{center}
for all $i=1,...,n$. If $f,g\in\Gamma$ are irreducible polynomials
we derive from these equalities that $f\equiv_{Y_i}g$ if and only if
the polynomials $s_i(f):=f(t_1,...,t_{i-1},t_i+1,t_{i+1},...,t_n)$
and $g$ are not coprime (i.e. the prime ideals of $\Gamma$ generated
by them are not coprime). Indeed we have that
$fY_i\Gamma=Y_is_i(f)\Gamma$ and  $\Gamma Y_i\Gamma =Y_i\Gamma$
using the above equalities. But then the obvious isomorphism of
'right' $\Gamma$-modules $\Gamma\cong Y_i\Gamma$ induces an
isomorphism

\begin{center}
$\frac{\Gamma
Y_i\Gamma}{fY_i\Gamma}=\frac{Y_i\Gamma}{Y_is_i(f)\Gamma}\stackrel{\cong}{\longleftrightarrow}\frac{\Gamma}{(s_i(f))}$.
\end{center}
It follows that $\frac{\Gamma Y_i\Gamma}{fY_i\Gamma +\Gamma
Y_ig}\cong\frac{\Gamma
Y_i\Gamma}{fY_i\Gamma}\otimes_\Gamma\frac{\Gamma}{(g)}$ is nonzero
if an only if
$\frac{\Gamma}{(s_i(f))}\otimes_\Gamma\frac{\Gamma}{(g)}\neq 0$.
This happens exactly when $s_i(f)$ and $g$ are not coprime.

We pass now to prove the statement. If $s_i(f)$ is not coprime with
$t_1-1$, for some $i=1,...,n$, then last paragraph applies with
$g=t_1-1$. So we assume that $s_i(f)$ is coprime with $t_1-1$ for
all $i=1,...,n$. (\emph{Note that this situation can actually
happen. For instance if $f=a+b(t_1-2)^m$,  with $m>0$ $a,b\in K$ and
$a\neq 0\neq a+(-1)^mb$}). We then put $f':=s_1(f)$ and express it
as a sum $\sum_{0\leq k\leq r}g_k(t_2,...,t_n)(t_1-1)^k$. Then we
get

\begin{center}
$\Gamma =f'\Gamma+(t_1-1)\Gamma=g_0\Gamma +(t_1-1)\Gamma$
\end{center}
and it is easy to derive from this that $g_0$ is a constant
polynomial, so that we can rewrite

\begin{center}
$f'(t_1,...,t_n)=a+(t_1-1)^mg(t_1,...,t_2)$,
\end{center}
where $g\in\Gamma\setminus\{0\}$ and $a\in K\setminus\{0\}$. Note
that, given any index $i=2,\ldots,n$,  we cannot have
$g(t_1,\ldots,t_{i-1},\alpha ,t_{i+1},\ldots,t_n)=0$, for all
$\alpha\in K$. Indeed in that case  the polynomial $g$ would be
zero. We then choose $\alpha\in K$ such that
$g(t_1,\alpha,t_3,\ldots,t_n)\neq 0$ and claim that $f'$ and
$t_2-\alpha$ are not coprimes. To see that, note that $f'$ and
$t_2-\alpha$ are coprime if, and only if,
$\bar{f}':=f'+(t_2-\alpha )$ is invertible in $\Gamma/(t_2-\alpha
)$. Using the canonical isomorphism

\begin{center}
$K[t_1,\ldots,t_n]/(t_2-\alpha)\stackrel{\cong}{\longrightarrow}K[t_1,t_3,\ldots,t_n]$

($\bar{h}\rightsquigarrow h(t_1,\alpha,t_3,\ldots,t_n)$),
\end{center}
we immediately find a polynomial $u\in K[t_1,t_3,\ldots,t_n]$
satisfying the equality

\begin{center}
$f'(t_1,\alpha,t_3,\ldots,t_n)u(t_1,t_3,\ldots,t_n)=1$
\end{center}
in $K[t_1,t_3,\ldots,t_n]$. It follows that $$f'(t_1,\alpha
,t_3,\ldots,t_n)=1+(t_1-1)^{m}g(t_1,\alpha ,t_2,\ldots,t_n)$$ is a
constant polynomial, something which can only happen when
$g(t_1,\alpha ,t_2,\ldots,t_n)=0$. But this contradicts our choice
of $\alpha$.

Put now $h:=t_2-\alpha$. We then get that $f\equiv_{Y_1}h$ since
$f'=s_1(f)$ is not coprime with $h=t_2-\alpha$. On the other hand,
we also have $h\equiv_{Y_2} t_1-1$ since
$s_2(h)=h(t_2+1)=t_2+1-\alpha$ is not coprime with $t_1-1$. We then
conclude that $f\equiv t_1-1$ as desired.
\end{proof}

We end the section with a result on extensions of $U$-modules.

\begin{proposition} \label{extension_groups}
Let $M$ and $N$ be nonzero $U$-modules and suppose that
$\frac{\Gamma u\Gamma}{\mathbf{q}u\Gamma +\Gamma u\mathbf{p}}=0$,
for all $u\in U$, $\mathbf{q}\in \Ass (M)$ and $\mathbf{p}\in \Ass (N)$.
The following assertions hold:

\begin{enumerate}
\item $\Ext_\Gamma^i(M,N)=0=\Ext_\Gamma^i(N,M)$, for all $i\geq 0$
\item $\Ext_U^1(N,M)=0$
\end{enumerate}
\end{proposition}
\begin{proof}
1) By taking $u=1$ above, we see that $\mathbf{p}$ and $\mathbf{q}$
are coprime whenever $\mathbf{p}\in \Ass (M)$ and $\mathbf{q}\in
\Ass (N)$. The assertion follows from Proposition \ref{coprime_assassins}.

2) Let $0\rightarrow M\longrightarrow X\longrightarrow N\rightarrow
0$ be an exact sequence in $U-\Mod.$ By assertion 1 we know that it
split in $\Gamma -\Mod$. Then we shall identify $X=M\oplus N$, in
which case the external multiplication map $U\times X\longrightarrow
X$ ($(u,x)\rightsquigarrow u\cdot x$) is entirely determined by the
$U$-module structures on $M$ and $N$ and by a $K$-bilinear map $\mu
:U\times N\longrightarrow M$ satisfying the following three
properties for all $u,u'\in U$, $g\in\Gamma$ and $y\in N$:

\begin{enumerate}
\item $\mu (uu',y)=u\mu (u',y)+\mu (u,u'y)$ (this guarantees that $(uu')\cdot y=u\cdot (u'\cdot y)$)
\item $\mu (g,y)=0$, for all $g\in\Gamma$ (this guarantees that the structure of $\Gamma$-module on $M\oplus N$ given by
restriction of scalars via the inclusion $j:\Gamma\hookrightarrow U$
is that of the direct sum)
\item $u\cdot y=\mu (u,y)+uy$ (this guarantees that the projection $\begin{pmatrix}0 & 1 \end{pmatrix}:X=M\oplus N\longrightarrow
N$ is a $U$-homomorphism)
\end{enumerate}
It follows  that the assignment $u\otimes y\rightsquigarrow\mu
(u,y)$  defines a homomorphism of $\Gamma$-modules $\mu '
:U\otimes_\Gamma N\longrightarrow M$.

We claim that $\mu '=0$. Suppose  not and take $\mathbf{q}\in
\Ass (\Im (\mu '))\subseteq \Ass (M)$. The surjective
$\Gamma$-homomorphism $U\otimes_\Gamma N\twoheadrightarrow \Im (\mu
')$ induces another surjective $\Gamma$-homomorphism

\begin{center}
$\oplus_{u\in U,y\in N}\Gamma u\Gamma\otimes_\Gamma\Gamma
y\twoheadrightarrow \Im (\mu ')$.
\end{center}
In particular, we get that $\mathbf{q}\in \Supp (\Gamma
u\Gamma\otimes_\Gamma\Gamma y)$, for some $u\in U$ and $y\in N$.
Since $\Gamma u\Gamma\otimes_\Gamma\Gamma y$ is an epimorphic image
in $\Gamma -\Mod$ of $\Gamma u\Gamma$, which is finitely generated as
'left' $\Gamma$-modules, it follows that $\Gamma
u\Gamma\otimes_\Gamma\Gamma y$ is a finitely generated
$\Gamma$-module and thereby that $\mathbf{q}(\Gamma
u\Gamma\otimes_\Gamma\Gamma y)\neq \Gamma
u\Gamma\otimes_\Gamma\Gamma y$. That means that the left arrow in
the exact sequence

\begin{center}
$\mathbf{q}u\Gamma\otimes_\Gamma\Gamma y\longrightarrow\Gamma
u\Gamma\otimes_\Gamma\Gamma y\longrightarrow\frac{\Gamma
u\Gamma}{\mathbf{q}u\Gamma}\otimes_\Gamma\Gamma y\rightarrow 0$
\end{center}
is not surjective, and hence that $\frac{\Gamma
u\Gamma}{\mathbf{q}u\Gamma}\otimes_\Gamma\Gamma y\neq 0$. The
argument of Lemma \ref{crucial_lemma_on_assassins}(1) shows that
there exists a $\mathbf{p}\in \Ass (\Gamma y)\subseteq \Ass (N)$ such
that $\frac{\Gamma
u\Gamma}{\mathbf{q}u\Gamma}\otimes_\Gamma\frac{\Gamma}{\mathbf{p}}\neq
0$. We then get $\frac{\Gamma u\Gamma}{\mathbf{q}u\Gamma +\Gamma
u\mathbf{p}}\neq 0$, which contradicts the hypothesis.
\end{proof}

\section{Applications and some open questions}

We start with a proposition which will be useful in the sequel for
its hypotheses are satisfied by all examples of this final section.

\begin{proposition} \label{final proposition}
In the setup \ref{setup}  suppose  in addition that the following
conditions hold:

\begin{enumerate}
\item If $Z=Z(U)$ is the center  of $U$ then $Z\cap\Gamma$
is equidimensional (see \cite{mat}, p. 250) \item $\Gamma$ is flat
as a $Z\cap\Gamma$-module
\item For each simple $U$-module, the endomorphism algebra
$End_U(M)$ has dimension equal to $1$ as a $K$-vector space.
\end{enumerate}
If  $U-fl$ denotes the subcategory of $U$-modules of finite length,
then $\mathcal{T}_i\cap U-fl=\mathcal{T}_j\cap U-fl$, for all
$i,j\geq Kdim(\Gamma )-Kdim(Z\cap\Gamma )$.
\end{proposition}
\begin{proof}
 Let $M$ be a simple
$U$-module. Then the structural map $K\longrightarrow End_U(M)$ is
an algebra isomorphism, which we view as an identification. On the
other hand, every element $z\in Z$ induces by multiplication an
endomorphism  $\lambda_z\in End_U(M)$. Put $Z'=Z\cap\Gamma$. The
assignment $z\rightsquigarrow\lambda_z$ gives then an isomorphism

\begin{center}
$Z'/ann_{Z'}(M)\stackrel{\cong}{\longrightarrow}End_U(M)=K$,
\end{center}
thus showing that $\mathbf{m}:=ann_{Z'}(M)$ is a maximal ideal of
$Z'$. Let now $\mathbf{p}\in Spec(\Gamma )$ be minimal over
$\Gamma\mathbf{m}$. We clearly have $\mathbf{m}=Z'\cap\mathbf{p}$
and we have an equality

\begin{center}
$ht(\mathbf{p})=ht(\mathbf{m})+Kdim(\frac{\Gamma_\mathbf{p}}{\Gamma_\mathbf{p}\mathbf{m}})$
\end{center}
(cf. \cite{mat}[Theorem 15.1]). But the prime spectrum of
$\frac{\Gamma_\mathbf{p}}{\Gamma_\mathbf{p}\mathbf{m}}$ is in
bijection with the set of $\mathbf{q}\in Spec(\Gamma )$ such that
$\Gamma\mathbf{m}\subseteq\mathbf{q}\subseteq\mathbf{p}$. By our
choice of $\mathbf{p}$, this implies that
$Spec(\frac{\Gamma_\mathbf{p}}{\Gamma_\mathbf{p}\mathbf{m}})$ has
one element. It follows that
$Kdim(\frac{\Gamma_\mathbf{p}}{\Gamma_\mathbf{p}\mathbf{m}})=0$, so
that $ht(\mathbf{p})=ht(\mathbf{m})$, for all $\mathbf{m}\in
Specm(Z')$ and all  $\mathbf{p}\in Spec(\Gamma )$ minimal over
$\Gamma\mathbf{m}$.

Put $d:=Kdim(\Gamma)$ and $e:=Kdim(Z')$. Equidimensionality of $Z'$
gives that $ht(\mathbf{m})=e$ (cf. \cite{Ku}[Corollary II.3.6]).
Then from the last paragraph and the inequality

\begin{center}
$ht(\mathbf{p})+Kdim(\Gamma /\mathbf{p})\leq Kdim(\Gamma )$
\end{center}
we readily derive that

\begin{center}
$Kdim(\frac{\Gamma}{\Gamma\mathbf{m}})=Sup\{Kdim(\Gamma/\mathbf{p}):$
$\mathbf{p}\in Spec(\Gamma )\text{ minimal over
}\Gamma\mathbf{m}\}\leq d-e$.
\end{center}
This says that the coheight of any $\mathbf{p}\in Spec(\Gamma )$
containing a maximal ideal of $Z'$ is always $\leq d-e$. In
particular that happens for all $\mathbf{p}\in Ass(M)$, for every
simple $U$-module $M$. It follows that the simple $U$-modules in
$\mathcal{T}_i$ are the same for all  $d-e\leq i\leq d$, which
implies the statement.
\end{proof}

\begin{remark}
Bearing in mind that our field is algebraically closed, condition
(3) in Proposition \ref{final proposition} is satisfied whenever $U$
admits an exhaustive filtration $U_0\subset U_1\subset...$ such that
the associated graded algebra $gr(U)$ is a commutative finitely
generated algebra (cf. \cite{Dix}[Lemma 2.6.4]). It is the case for
all finite $W$-algebras  (cf. \cite{BK1},Theorem 10.1 or \cite{GG},4.4).
\end{remark}

The following problems are of special interest in the case of
enveloping algebras of Lie algebras and finite W-algebras.

\begin{problems}\label{problems}
Suppose that $\Gamma$ and $U$ satisfy the conditions of Setup
\ref{setup} and also the hypotheses of Proposition \ref{final
proposition}. We propose the following problems:
\begin{enumerate}
\item To identify the set $\mathbf{N}_U$ of natural numbers $0\leq
j\leq d-e$ for which there exists a simple $U$-module $M$ such
that $t_j(M)=M$ and $t_{j-1}(M)=0$ (convening that $t_{-1}(M)=0$).
\item Given $j\in\mathbf{N}_U$, to identify the set of
$\mathbf{p}\in Spec(\Gamma )$ such that $cht(\mathbf{p})=j$ and
$\mathbf{p}\in Ass(M)$, for some simple $U$-module $M$ \item
(Local version) Given a character $\chi
:Z'=Z\cap\Gamma\longrightarrow K$, to identify the set
$\mathbf{N}(\chi )$ of natural numbers $0\leq j\leq d-e$ for which
there exists a simple $U$-module $M$ annihilated by $Ker(\chi )$
with $t_j(M)=M$ and $t_{j-1}(M)=0$. For any  $j\in\mathbf{N}(\chi
)$, to identify all $\mathbf{p}\in Spec(\Gamma)$ such that
$cht(\mathbf{p})=j$, $Ker(\xi )\subset \mathbf{p}$ and
$\mathbf{p}\in Ass(M)$ for some simple $U$-module $M$.
\end{enumerate}
\end{problems}

We move  now to the announced   classical examples.

\subsection{Finite W-algebras}
 Associated with a nilpotent element and a good grading in the Lie
 algebra $gl_n$, there is associated a finite $W$-algebra (see \cite{EK} for the definition and
 details). Each finite W-algebra of type $\mathbf{A}$ is determined
 by a sequence of integers $\tau =(p_1,\ldots, p_m)$ such that $1\leq p_1\leq \ldots\leq
 p_m$ and $p_1+\ldots+p_m=n$. We denote such an algebra by $W(\tau )$.
 If for each $k=1,\ldots, m$ we put $\tau_k=(p_1,\ldots, p_k)$, then we
 obtain a chain of subalgebras

 \begin{center}
 $W(\tau_1)\subset \ldots\subset W(\tau_m)=W(\tau )$.
 \end{center}
 The subalgebra $\Gamma$ of $W(\tau )$ generated by the centers of
 the $W(\tau_k)$ is a commutative algebra usually called the
 \emph{Gelfand-Tselin subalgebra} of $W(\tau )$.

 As shown in \cite{fo-Ga1} and \cite{fo-Ga2},  the algebra $U=W(\tau)$ and the
 commutative subalgebra $\Gamma$ satisfy all the
 conditions of Setup \ref{setup} and all the hypothesis of Proposition \ref{final proposition}, actually with $Z\subset\Gamma$ and hence $Z\cap\Gamma =Z$.  Moreover, we have $d=mp_1+(m-1)p_2+...+2p_{m-1}+p_m$ and
 $e=p_1+\ldots+p_m$ (see \cite{FMO} and \cite{BK1}), where $d$ and $e$ are as in Proposition \ref{final proposition}.
 In particular we get:

 \begin{corollary} \label{application_to_finite_W-algebras}
 Let us consider the natural number $r=(m-1)p_1+(m-2)p_2+\ldots+p_{m-1}$.
 The following assertions hold:

 \begin{enumerate}
\item The torsion theories $(\mathcal{T}_i,\mathcal{F}_i)$
($i=0,1,\ldots, d$) are liftable from $\Gamma -\Mod$ to $W(\tau
)-\Mod$. \item If $M$ is a simple $W(\tau )$-module then there is
a unique natural number $0\leq j\leq r$ such that $t_j(M)=M$ and
$t_{j-1}(M)=0$. In this case all prime ideals in $\Ass (M)$ have
coheight exactly $j$.
\end{enumerate}
 \end{corollary}

Note that in the case $m=n$ and $p_1=\ldots=p_m=1$ the
corresponding $W$-algebra is isomorphic to $U(\gl_n)$.

\subsection{The Lie algebra $\gl_n$}
Given any positive integer $n$ and any basis $\pi =\{\alpha_1,
\ldots,\alpha_n\}$ of the root system of the Lie algebra $\gl_n$,
we denote by $\gl_i$ the Lie subalgebra corresponding to the
simple roots $\alpha_1,\ldots, \alpha_i$. We then have  inclusions
of Lie algebras

\begin{center}
$\gl_1\subset \gl_2\subset \ldots\subset \gl_n$
\end{center}
inducing corresponding inclusions of associative algebras

\begin{center}
$U_1\subset U_2\subset \ldots\subset U_n$,
\end{center}
where $U_k=U(\gl_k)$ is the universal enveloping algebra of $\gl_k$
for each $k>0$. If we put $U=U_n$ then the subalgebra $\Gamma (\pi
)$ of $U$ generated by the centers of  $U_1, \ldots, U_n$ is a
maximal commutative subalgebra, called the \emph{Gelfand-Tsetlin
subalgebra} of $U$ associated to the root system $\pi$. The
inclusion $\Gamma (\pi )\subset U$ satisfies all the requirements of
Setup \ref{setup} and the hypotheses of Proposition \ref{final
proposition}, again with $Z\subseteq\Gamma$. Concretely $\Gamma (\pi )$ is isomorphic to a
polynomial algebra on $\frac{n(n+1)}{2}$ variables (cf.
\cite{fo-Ga1}, \cite{fo-Ga2}) while the center $Z=Z(U)$ is a
polynomial algebra on $n$ variables. We therefore have:

\begin{corollary} \label{application_to_finite_W-algebras1}
 The following assertions hold:

 \begin{enumerate}
\item The torsion theories $(\mathcal{T}_i,\mathcal{F}_i)$
($i=0,1,\ldots, \frac{n(n+1)}{2}$) are liftable from $\Gamma
(\pi)-\Mod$ to $U(\gl_n)-\Mod$. \item If $M$ is a simple
$\gl_n$-module then there is a unique natural number $0\leq j\leq
\frac{n(n-1)}{2}$ such that $t_j(M)=M$ and $t_{j-1}(M)=0$. In this
case all prime ideals in $\Ass (M)$ have coheight exactly $j$.
\end{enumerate}
 \end{corollary}

 An interesting phenomenon for $U_n=U(gl_n)$ is that there are
several
 Gelfand-Tsetlin subalgebras to which we can apply our general
 theory, namely, one per each choice of a basis of the root system. We  denote by ${\mathcal{T}}_i(\pi)$ the
 class of
 $U_n$-modules $M$ such that, viewed as $\Gamma (\pi)$-module, $M$
belongs to $\mathcal{T}_i$. Since different root systems are
 conjugated by the Weyl group, one immediately gets:

 \begin{proposition} \label{different_root_systems}
 Let $\pi$ and $\pi'$ be two bases of the root systems of $\gl_n$. The categories
 ${\mathcal{T}}_i(\pi)$ and ${\mathcal{T}}_i(\pi')$ are
 equivalent for any $i$.
\end{proposition}

Concerning Problem \ref{problems}(1), it is well-known that
$0\in\mathbf{N}_U$   when $U$ is a finite $W$-algebra of type
$\mathbf{A}$. For the particular case $U=U(gl_n)$ we  have that
$1\in\mathbf{N}_U$, as the following example show.

 \begin{example}There are simple $\gl_n$-modules which are  not
in ${\mathcal{T}}_0$ for all $n>1$.
\end{example}
\begin{proof}
  Consider any generic simple non-weight (with respect to any Cartan subalgebra) $\gl_2$-module
  $V$, such modules exist by \cite{bl}. Then
$V\in {\mathcal{T}}_1$ and is not Gelfand-Tsetlin. Let $H$ be a
Cartan subalgebra of $\gl_3$. Fix $a\in \mathbb C$. Let
$(c_1,c_2)$ be the central character of $V$ ($c_1$ is an
eigenvalue of $e_{11}+e_{22}$ and $c_2$ is an eigenvalue of the
quadratic Casimir element). Let $\mathfrak P$ be a parabolic
subalgebra of $\gl_3$ whose Levi factor is $\gl_2+H$. Now consider
the induced module $M(V, a)=U(\gl_3)\otimes_{U(\mathfrak P)} V$
where $V$ is naturally viewed as a $\mathfrak P$-module with a
trivial action of the radical and $e_{11}+e_{22}+e_{33}$ acts by
multiplication by $a$. Then $M(V, a)$ has a unique simple quotient
$L(V, a)$ which belongs to the subcategory ${\mathcal{T}}_1\subset
\gl_3-\Mod$ and is not Gelfand-Tsetlin. Similarly, one can induce
now from $L(V, a)$ to get a $\gl_4$-module with a unique simple
quotient in ${\mathcal{T}}_1\subset \gl_4-\Mod$ which is not
Gelfand-Tsetlin. One continues inductively. Hence, for each $n\geq
2$ we construct a simple $\gl_n$-module in ${\mathcal{T}}_1$ which
is not Gelfand-Tsetlin.
\end{proof}

 \section{Acknowledgment}
The first  author is supported in part by the CNPq grant (301743/
2007-0) and by the Fapesp grant (2005/ 60337-2). The third author is
supported by research projects from the D.G.I. of the Spanish
Ministry of Education and the Fundaci\'{o}n S\'{e}neca of Murcia
(Grupos de Excelencia), with a part of FEDER funds. This work
 was done during a visit of the first author to the
Universidad de Murcia, funded by the Fundaci\'{o}n S\'{e}neca, and a
sabbatical visit of the third author to the Universit\'{e} de Paris
7 funded by the Spanish Ministry of Education. Both authors thank
the support of the institutions and the hospitality of the
respective Universities.

\end{document}